\documentclass[11pt,twoside, final]{amsart}
\copyrightinfo{0}{Iranian Mathematical Society}
\pagespan{1}{\pageref*{LastPage}}
\usepackage{etoolbox,lastpage}
\commby{}
\date{\scriptsize   Received: , Accepted: .}
\usepackage{amsmath,amsthm,amscd,amsfonts,amssymb,enumerate}
\usepackage{graphicx}		
\usepackage{color}
\usepackage[colorlinks]{hyperref}
 
\newtheorem{theorem}{Theorem}[section]
\newtheorem{proposition}[theorem]{Proposition}
\newtheorem{lemma}[theorem]{Lemma}
\newtheorem{corollary}[theorem]{Corollary}
\theoremstyle{definition}
\newtheorem{definition}[theorem]{Definition}
\newtheorem{example}[theorem]{Example}
\theoremstyle{remark}

\numberwithin{equation}{section}
 
\begin{document}

 
\title[Some notes on $b$-weakly compact operators]{Some notes on $b$-weakly compact operators} 
\author[M.  Mousavi Amiri]{Masoumeh Mousavi Amiri}
\address[Masoumeh Mousavi Amiri]{Department of Mathematics, 
University of Mohaghegh  Ardabili, 
Ardabil, Iran.}
\email{masoume.mousavi@gmail.com}
\author[K. Haghnejad Azar]{Kazem Haghnejad Azar $^*$}
\address[Kazem Haghnejad Azar]{Department of Mathematics, 
University of Mohaghegh  Ardabili, 
Ardabil, Iran.}
\email{haghnejad@uma.ac.ir}


  \thanks{$^*$Corresponding author}
%
 
 \maketitle
%

\begin{abstract}
In this paper, we will study some properties of b-weakly compact operators and we will investigate their relationships to some variety of operators on the normed vector lattices. With some new conditions, we show that the modulus  of an operator $T$ from Banach lattice $E$ into Dedekind complete Banach lattice $F$ exists and is $b$-weakly operator whenever $T$ is a $b$-weakly compact operator. 
We show that every Dunford-Pettis operator from a Banach lattice $E$ into a Banach space $X$ is b-weakly compact,  and the converse  holds whenever  $E$ is an $AM$-space or the norm of $E^\prime$ is order continuous and $E$ has the Dunford-Pettis property. We also show that each order bounded operator from a Banach lattice into a $KB$-space admits a $b$-weakly compact modulus.\\
\textbf{Keywords:}  Banach lattice, order continuous norm, b-weakly compact operator, Dunford-Pettis operator.  \\
\textbf{MSC(2010):}  Primary 46B42; Secondary 47B60.
\end{abstract}

\section{Introduction}
An operator $T$ from a Banach lattice $E$ into a Banach space $X$ is said to be b-weakly compact, if it maps each subset of $E$ which is b-order bounded (i.e. order bounded in the topological bidual $E^{\prime\prime}$) into a relatively weakly compact subset of $X$.  Note that in \cite{2}, the class of b-weakly compact operators is introduced by  Alpay-Altin-Tonyali  and several interesting characterizations were given in \cite{2,4,9,10}.   After that, a series of papers which gave different characterizations of this class of operators were published \cite{2,3,4,5,6,7,8,9,10}. The most beautiful property of the class of b-weakly compact operators is that it satisfies the domination property as proved in \cite{2}.  But one of shortcomings of this class is that the modulus of a $b$-weakly compact operator need not be $b$-weakly compact. Note that each weakly compact operator is b-weakly compact, but the converse is not true in general. In \cite{8}, authors  characterized Banach lattices on which each b-weakly compact operator is weakly compact. In \cite{10},  authors proved that an operator $T$ from a Banach lattice $E$ into a Banach space $X$ is $b$-weakly compact if and only if $(Tx_n)$  is norm convergent for every positive increasing sequence $(x_n)$ of the closed unit ball  $B_E$ of $E$. 
\subsection{Some basic definitions}
 An element $e>0$ in a Riesz  space $E$ is said to be an order unit whenever for each $x\in E$ there exists some $\lambda >0$ with $|x| \leq \lambda e$.  For example $\ell^\infty$ has order unit, but $c_0$ has not.   A sequence $(x_n)$ in a vector lattice is said to be disjoint whenever $|x_n| \wedge |x_m| =0$ holds for $n\neq m$.   Let $E$ be a Riesz space. The subset $ E^+=\{x\in E:x\geqslant0\}$  is called the positive cone of $E$ and the elements of $E^+$ are called the positive elements of $E$. For an operator $T:E\rightarrow F$ between two Riesz spaces we shall say that its modulus $|T|$ exists ( or that $T$ possesses a modulus) whenever $|T|:=T\vee (-T)$ exists in the sense that $|T|$ is the supremum of the set $\{-T,T\}$ in $\mathcal{L}(E,F)$.
An operator  $T:E\rightarrow F$ between two Riesz spaces  is called order bounded if it maps order bounded subsets of $E$ into order bounded subsets of $F$.  An operator $T: E\rightarrow F$ between two Riesz spaces  is positive if $T(x)\geq 0$ in $F$ whenever $x\geq 0$ in $E$.  Note that each positive linear mapping on a Banach lattice is continuous.  It is clear that every positive operator is order bounded, but the converse is not true in general.       
A Banach lattice $E$ has order continuous norm if $\| x_\alpha\|\rightarrow 0$ for every decreasing net $(x_\alpha)_\alpha$ with $\inf_\alpha x_\alpha=0$.  If $E$ is a Banach lattice,  its topological dual $E^\prime$, endowed with the dual norm and dual order is also a Banach lattice.  A Banach lattice $E$ is  said to be an $AM$-space if for each $x,y\in E$ such that $|x|\wedge |y|=0$,  we have $\|x+y\|= max \{\|x\|, \|y\|\}$.   A Banach lattice $E$ is said to be $KB$-space whenever each increasing norm bounded sequence of $E^+$ is norm convergent.  A subset $A$ of a Riesz space $E$ is called b-order bounded in $E$ if it is order bounded in the bidual  $E^{\thicksim\thicksim}$.      An operator $T:E\rightarrow F$ between two Banach spaces is called a Dunford-Pettis operator whenever $x_n\xrightarrow{w} 0$ implies $Tx_n\xrightarrow{\|\cdot \|} 0$. Recall that an operator $T$ from a Banach lattice $E$ into a Banach space $X$ is said to be order weakly compact, if it maps each order bounded subset of $E$ into a relatively weakly compact subset of $X$, i.e., $T[-x,x]$ is relatively weakly compact in $X$ for each $x\in E^+$. Assume that $E$ and $F$ are normed lattice.
A positive linear operator $T:E\rightarrow F$ is called almost interval preserving if $T[0,x]$ is dense in $[0,Tx]$ for every $x\in E^+$.
Let $E$ be a vector lattice. A sequence  
 $\{x_n\}_1^ \infty\subset E$
 is 
 called  order  convergent  to 
 $x$
 as 
 $n \to  \infty$
 if there exists a sequence  
 $\{y_n\}^\infty_1$
 such  that  
 $y_n \downarrow 0$
 as
 $n  \rightarrow \infty$
 and 
 $|x_n -x| \leq  y_n$
 for  all 
 $n$. We will write $x_n\xrightarrow{o_1}x$ when $\{x_n\} $ is order convergent to $x$. A sequence $\{x_n\}$ in a vector lattice  $E$ is    strongly order convergent  to  $x\in E$, denoted by $x _n \xrightarrow{o_2}x$ whenever there exists a net $\{y_\beta\}_{\beta\in \mathcal{B}}$  in $E$ such that $y_\beta \downarrow 0$ and that for every $\beta\in \mathcal{B}$, there exists $n_0$ such that  $|x_n -x| \leq y_\beta$ for all $n\geq n_0$. It is clear that every order convergent sequence is strongly order convergent, but two convergence are different, for information see, \cite{1b}. A net $ (x_{\alpha})_{\alpha}$ in Banach lattice $E$ is unbounded norm convergent (or, $un$-convergent for short) to $ x \in E $ if $ | x_{\alpha} - x | \wedge u\xrightarrow{\|\cdot \|} 0$ for all $ u \in E^{+} $.  We denote this convergence by $ x_{\alpha} \xrightarrow{un}x $. This convergence has been introduced and studied in \cite{9d, 10a}. 
For terminology concerning Banach lattice theory and positive operators, we refer the reader to the excellent books  of \cite{1,11,12}.


\section{\bf Main Results}
The collections of   b-weakly compact operators,   order weakly compact operators,  weakly compact operators and compact operators will be denoted by
 $W_b(E,X)$, $W_o(E,X)$, $W(E,X)$ and $K(E,X)$, respectively,  whenever there is not confused. We have the following relationships between these spaces:
$$K(E,X) \subseteq W(E,X)\subseteq W_b(E,X)\subseteq  W_{o}(E,X).$$
In \cite{2}, authors show that the above inclusion may be proper. 
For example, note that each weakly compact operator is a $b$-weakly operator, but the converse may be false in general. Of course the identity operator $I:L^1 [0,1]\rightarrow L^1 [0,1]$ is a $b$-weakly operator, but is not weakly compact.
Now let  $E$ be a Banach lattice such that the norm of $E^ \prime $ is order continuous and let $X$ be a Banach space. Then,  by    \cite [Theorem 2.2]{8}, it is obvious that each $b$-weakly operator $T:E\rightarrow X$ is weakly compact.\\

The next  example due to Z. L. Chen and A. W. Wickstead in  \cite{13} shows that the order bounded $b$-weakly compact operators from a Banach lattice into a Dedekind complete Banach lattice do not form a lattice, i.e., a modulus of a $b$-weakly compact operator need not be $b$-weakly compact.

\begin{example} \label{2.13}
 Let $E=C[0,1]$, $F=l_\infty (F_n)$ where $F_n=(l_\infty,\|\cdot \|_n)$ and $\|(\lambda_k)\|_n=\max\{\|(\lambda_k)\|_\infty ,n\limsup(|\lambda_k|)\}$ for all $(\lambda_k)\in l_\infty$. Then for each $n\in \mathbb{N}$, $F_n$ is a Dedekind complete $AM$-space, hence so is $F$. Define $T_n:E\rightarrow F_n$ by $T_n(f)=(2^n.\int _{I_n} f.r_kdt)_{k=1}^{\infty}\in F_n$ for all $f\in E$, where $r_n$ is the $n^,$th Rademacher function on $[0,1]$ and $I_n=(2^{-n},2^{-n+1})$.\\
Now define $T:E\rightarrow F$ by $T(f)=(\frac{1}{n}T_n(f))_{n=1}^{\infty}$. Then $T$ is a weakly compact operator. So,  $T$ is a $b$-weakly compact operator and its modulus $|T|$ exists and $|T|$ is not order weakly compact hence not $b$-weakly compact.  So,  $W_{b}(E,F)$ is not a Riesz space. 
\end{example}
Alpay  and Altin in \cite{4} show that for $b$-weakly compact operator $T$ from a Banach lattice $E$ into a Dedekind complete Banach lattice $F$,  the modulus  of $T$ exist and is $b$-weakly compact operator whenever $F$ is $AM$-space with order unit. Now in the following theorems, we show that $T\in W_b(E,F)$ whenever $\vert T\vert\in W_b(E,F)$ and with some new conditions, we show that the modulus  of $T$ exists and is $b$-weakly operator whenever $T$ is a $b$-weakly compact operator.

\begin{theorem}\label{t:2.6}
Let $E$ and $F$ be normed vector lattices. We have the following assertions.
\begin{enumerate}
\item If  $T:E\rightarrow F$ is an order bounded operator and $F$ is  $KB$-space, then $T$ and $\vert T\vert$ are $b$-weakly operator.
\item If $\vert T \vert$ is a $b$-weakly compact operator, then
 \begin{equation*}
 T \in  W_{b}(E,F)
 \end{equation*}
 \end{enumerate}
 \end{theorem}
 \begin{proof}
 \begin{enumerate}
\item Since every $KB$-space has order continuous norm, so $F$ is a Dedekind complete Banach lattice. Then, by Theorem 1.18 from \cite{1}, $|T|$ exists. Now,  let $(x_n)$ be a positive increasing  sequence in $B_E$. By Theorem 4.3 of \cite{1}, $T^+(x_n)$ is a Positive increasing norm bounded sequence in $F$.  Since $F$ is a $KB$-space,   $T^+(x_n)$ is norm convergent. Thus, $T^+$ is $b$-weakly compact. Similarly, $T^-$ is $b$-weakly compact.
It follows from the identities $T=T^++T^-$ and  $|T|=T^++T^-$ that $T$ and $|T|$ are $b$-weakly compact operators, so we are done.
 \item Since $0 \leq T^- , T^{+}\leq \vert T \vert$, then by using Corollary 2.9 from \cite{2}, $T^-$ and $T^+$ are two $b$-weakly compact operators, so $T$ is a $b$-weakly compact operator.
\end{enumerate}  
 \end{proof}

\begin{theorem}
Let $T$ be an order bounded operator from Banach lattice $E$ into Dedekind complete Banach lattice $F$. If $c_0$ dose not embed in $F$, then $T$ and $\vert T \vert$ are $b$-weakly compact operators. 
\end{theorem}
\begin{proof}
At first, assume that $T$ is a positive operator. By Theorem 4.3 \cite{1}, $T$ is continuous. Thus by Proposition 1 \cite{4}, it suffices to show that $(Tx_n)$ is norm convergent to zero for each $b$-order bounded disjoint sequence $(x_n)$ in $E^+$. Now let $(x_n)$ be a $b$-order bounded and disjoint sequence in $E^+$. It follows that there is a $0\leq x^{\prime\prime}\in E^{\prime\prime}$ such that $0\leq x_n\leq x^{\prime\prime}$     for all $n$.  Then for each $0\leq x^\prime \in E^\prime$, we have 
\begin{align*}
x^\prime (\sum_{i=1}^k x_n)=x^\prime (\bigvee_{i=1}^k x_n)\leq x^{\prime\prime}(x^\prime), \quad\text{for~ all}~ k~ \text{holds.}
\end{align*}
 Hence $x^\prime (\sum_{i=1}^\infty x_n)<\infty$ for all $0\leq x^\prime \in E^\prime$. Then for each $0\leq y^\prime\in F^\prime$ we have $\sum_{i=1}^k y^\prime (T x_n)=\sum_{i=1}^k T^\prime y^\prime (x_n)<+\infty$. 
It follows that the sequence $(s_m=\sum_{n=1}^m Tx_n)$ is weakly bounded, and by Theorem 2.5.5 \cite{Megginson}, it is norm bounded.   Since $c_0$ dose not embed in $F$, by Theorem 4.60 \cite{1}, $F$ is a $KB$-space. Since the sequence $(s_m)$ is positive, increasing  and norm bounded, so it is norm convergent, and so the sequence $(\sum_{n=m}^k Tx_n)$
is norm convergent to zero. It follows that $\lim \Vert Tx_n\Vert=0$. Therefore, $T$ is a $b$-weakly compact operator. Now, let $T$ be an order bounded operator. By the identities $T=T^+-T^-$ and $|T|=T^++T^-$, it follows that $T$ and $|T|$ are $b$-weakly compact operators.
\end{proof} 

\begin{proposition} \label{2.1}
Let $E$ and $F$ be two Banach lattices. Then we have the following assertions.
\begin{enumerate}
\item  If $T$ is a positive and order weakly compact operator from $E$ onto $F$, then $F$ has $\sigma$-order continuous norm.\label{2.1.a}
\item If a positive $b$-weakly compact operator $T:E \rightarrow F$ between Banach lattices is surjective, then the norm of $F$ is $\sigma$-order continuous.
\end{enumerate}
\end{proposition}
\begin{proof}
\begin{enumerate}
\item  Assume that $\{y_n\}_n\subseteq F$ with $y_n\downarrow 0$. Since $T$ is surjective, there is an element $x_1\in E$ such that $Tx_1=y_1$. It is clear that $\{y_n\}_n\subseteq [0,y_1]\subseteq T([0,x_1])$. Since $T([0,x_1])$ is relatively  weakly compact, there is a subsequence $\{y_{n_j}\}_j$ of  $\{y_n\}_n$ such that $y_{n_j} \xrightarrow{w} y_0 \in F$. Since $\{y_{n_j}\}_j$ is a decreasing sequence, by  Theorem 3.52 from \cite{1}, $y_{n_j} \xrightarrow{\|\cdot \|} y_0 \in F$.  It follows from  $y_n\downarrow 0$ that \ $y_0=0$. Hence $\|y_n\| \rightarrow 0$.  Thus $F$ has order continuous norm.
\item Follows from (\ref{2.1.a}).
\end{enumerate}
\end{proof}

Now by \cite[Theorem 4.11]{1} we have the following result:
 \vspace{0.2cm}
\begin{corollary} \label{2.3}
 Let $E$ be a normed lattice with order continuous norm.  Then the norm completion of $E$ has order continuous norm.
\end{corollary}
\begin{proof}
We prove that  Theorem 4.11 (2)  of  \cite{1} holds. Let $0\leq x_n\uparrow\leq x$ holds in $E$.
It follows from \cite[Corollary 4.10]{1} that $E$ is Dedekind complete.
Put $y=\sup{x_n}$. So, $y-x_n\downarrow 0$ in $E$. Therefore, $\| y-x_n\|\to 0$. We have
$$\|x_n-x_m\|\leq \|x_n-y\|+\| y-x_m\| \to 0,$$
hence $\{x_n\}$ is a norm  Cauchy sequence.
\end{proof}

\vspace{0.2cm}

\begin{proposition} \label{2.5}
Let $E$ and $F$ be two Banach lattices such that $E$ has order unit and $F$ has order continuous norm. Then every order bounded operator $T : E \rightarrow F$ is b-weakly compact.
\end{proposition}
\begin{proof}
Let $A$ be a b-order bounded subset of $E$. Since $E$ has an order unit, $A$ is order bounded in $E$.  Therefore, $T(A)$ is order bounded in $F$. Since $F$ has an order continuous norm, $T(A)$ is  a relatively weakly compact. Thus $T$ is b-weakly compact operator.  
\end{proof}

\begin{example} \label{1.1, 2.7}
 Every order bounded operator from $\ell^ \infty$ into $c_0$ is b-weakly compact.
\end{example}

\begin{theorem}
Let $E$ and $F$ be two Banach lattices where $E$ has order continuous norm. Let $G$ be an order dense sublattice of $E$ and let $T$ be a positive operator from $E$ into $F$.  If $T\in W_b(G,F)$, then $T\in W_b(E,F)$.
\end{theorem} 
\begin{proof}
 Let $\{x_n\}$ be a positive increasing sequence in $E$ with $\sup_n\Vert x_n \Vert<\infty$. Since $G$ is order dense in $E$, by Theorem 1.34 from \cite{1}, we have
\begin{equation*}
\{y\in G:~0\leq y\leq x_n\}\uparrow x_n,
\end{equation*}
for each $n$. Let $\{y_{mn}\}_{m=1}^\infty\subset G$ with $0\leq y_{mn}\uparrow x_n$ for each $n$. Put $z_{mn}=\vee_{i=1}^n y_{mi}$ and  $0\leq T\in L(E,F)$. It follows that $z_{mn}\uparrow_m x_n$ and $\sup_{m,n}\Vert z_{mn}\Vert\leq \sup_n\| x_n\|<\infty$. Now, if $T\in W_b(G,F)$, then $\{Tz_{mn}\}$ is norm convergent to some point $y\in F$. Now, we have the following inequalities
\begin{equation*}
\|Tx_n-Tz_{mn}\|\leq\|T\|\|x_n-z_{mn}\|\leq \|T\|\|x_n-y_{mn}\|\rightarrow 0.
\end{equation*}
Thus by the following inequality proof holds
 \begin{equation*}
\|Tx_n-y\|\leq\|Tx_n-Tz_{mn}\|+\|Tz_{mn}-y\|.
\end{equation*}
\end{proof} 

\begin{definition}
Let $E$ and $F$ be two vector lattices.  $L^{(1)}_c(E,F)$ (resp. $L^{(2)}_c(E,F)$) is the collection of operators $T\in L_b(E,F)$, which $x_n\xrightarrow{o_1}0$ ($x_n\xrightarrow{o_2}0$) \text{implies} $Tx_{n_k} \xrightarrow{o_1}0$ (resp. $Tx_{n_k} \xrightarrow{o_2}0$)
\text{whenever} $\{x_{n_k}\}$ is a subsequence of $\{x_{n}\}$.
\end{definition}
In \cite{1b}, there are some examples which shows that two classifications of operators $L^{(1)}_c(E,F)$ and $L^{(2)}_c(E,F)$ are different.

\begin{theorem}\label{2.8a}
Let $E$, $F$ be two Banach lattices such that $E$ has order continuous norm. Then
\begin{enumerate}
\item  $W_b(E,F)^+ \subseteq L^{(2)}_c(E,F)$.
\item If $W_b(E,F)$  is vector lattice and  $F$  Dedekind complete, then $W_b(E,F)$ is an order ideal of $ L^{(1)}_c(E,F)= L^{(2)}_c(E,F)$.
\end{enumerate}
\end{theorem}
\begin{proof}
\begin{enumerate}
\item Let $T$ be a positive $b$-weakly compact operator and let $\{x_n\}\subset E$ be a strongly order convergence sequence in $E$.
Without lose of generality, we set $0\leq x_n\xrightarrow{o_2} 0$, which follows $\{x_n\}$ is norm convergent to zero. Set $\{x_{n_j}\}$ as subsequence with $\sum_{k=1}^{+\infty}\Vert x_{n_j} \Vert< +\infty$. Define $y_m=\sum_{j=1}^m  x_{n_j}$. Then $0\leq y_m\uparrow$ and $\sup_m\Vert y_m\Vert<\infty$. Since $T$ is $b$-weakly compact operator, $\{Ty_m\}$ is norm convergent to some point $z\in F$. Now by \cite{9g}, page 7, it has a subsequence as $\{Ty_{m_k}\}$ which  is strongly order convergent to $z\in F$. Thus there is $\{z_\beta\}\subset F^+$ and that for each $\beta$ there exists $n_0$ $\vert Ty_{m_k}-z\vert\leq z_\beta\downarrow 0$  whenever  $k\geq n_0$. If we set $k\geq k' \geq n_0$, then  we have the following inequalities
\begin{align*}
0&\leq Tx_{n_{m_k}}\leq \vert Ty_{m_k}-Ty_{m_{k'}}\vert\\
&\leq\vert Ty_{m_k}-z\vert +\vert Ty_{m_{k'}}-z\vert\leq z_\beta+z_\beta\downarrow 0, 
\end{align*}
which shows that $T\in L^{(2)}_c(E,F)$ and proof immediately follows. 
 \item By equality $T=T^+-T^-$ and Theorem 1.7 from \cite{1b}, we have $ L^{(2)}_c(E,F)=L^{(1)}_c(E,F)$.
Since $W_b(E,F)$ is a vector lattice, it follows from part (1) that $W_b(E,F)$ is a subspace of $ L^{(1)}_c(E,F)$. Now proof follows from the fact that $W_b(E,F)$ satisfies the domination property.
\end{enumerate}
\end{proof}
By using conditions of Theorem \ref{2.8a}, we can design the following question.\\
\noindent {\bf Question.}  Is  $W_b(E,F)$ a band in $ L^{(1)}_c(E,F)=L^{(2)}_c(E,F)$?

\begin{proposition} \label{2.6}
Let $E$ and $F$ be two Banach lattices such that $F$ is a $KB$-space.
Then every bounded operator $T : E \rightarrow F$ is b-weakly compact.
\end{proposition}
\begin{proof}
By using  \cite[Proposition 1]{4},  it is enough to show that $\{Tx_n\}_n$ is norm convergent for each b-order bounded increasing sequence $\{x_n\}_n$ in $E^+$. Let $\{x_n\}_n$ be a b-order bounded increasing sequence in $E^+$. Since $F$ is a $KB$-space, by \cite[Theorem 4.63]{1}, there exists a  $KB$-space $G$, a lattice homomorphism $R : E \rightarrow G$ and an operator $S : G \rightarrow F$ such that $T=S\circ R$. Since $R$ is a lattice homomorphism,  $R$ is a positive operator and therefore  is b-order bounded. Then $R(x_n)$ is a b-order bounded increasing sequence in $G$. Since $G$ is a $KB$-space, $R(x_n)$ is norm convergent in $G$. It follows that  $S\circ R(x_n)$ is also norm convergent in $F$. Hence $\{T(x_n)\}$ is norm convergent in $F$. This completes  the proof.
\end{proof}

\begin{example} \label{2.7}
 Let $E$ be a Banach lattice. Then every bounded operator from $E$ into $\ell^1$ is b-weakly compact.
\end{example}

\vspace{0.2cm}
In the following proposition,  we show that each Dunford-Pettis operator is b-weakly compact,  the converse is not always true. In fact, the identity operator of the Banach lattice $\ell^2$ is b-weakly compact, but it is not Dunford-Pettis.
Recall that if $E$ is a Banach lattice and if $0\leqslant{ x^{\prime\prime}} \in{ E^{\prime\prime}}$, then the principal ideal $I_{x^{\prime\prime}}$ generated by  $x^{\prime\prime} \in E^{\prime\prime}$ under the norm ${\|\cdot\|}_\infty$ defined by
$$\| y^{\prime\prime}\|_\infty =\inf \{\lambda >0: ~| y^{\prime\prime}| \leq \lambda x^{\prime\prime}\},\      y^{\prime\prime} \in {I_{x^{\prime\prime}}},$$
is an $AM$-space with unit $x^{\prime\prime}$, which its closed unit ball coincides with the order interval $[-x^{\prime\prime},x^{\prime\prime}]$.

\begin{lemma} \label{2.8}
Let $E$ be a Banach lattice. Then every b-order bounded disjoint sequence in $E$ is weakly convergent to zero.
\end{lemma}
\begin{proof}
Let $\{x_n\}_n$ be a disjoint sequence in $E$ such that $\{x_n\}_n \subseteq{[-x^{\prime\prime},x^{\prime\prime}]}$ for some $x^{\prime\prime} \in E^{\prime\prime}$. Let $Y=I_{x^{\prime\prime}} \cap E$ and equip $Y$ with the order unit norm ${\|\cdot\|}_\infty$ generated by $x^{\prime\prime}$. The space $(Y,{\|\cdot\|}_\infty)$ is an $AM$-space, so,  $Y^\prime$ is an $AL$-space and hence its norm is order continuous.  Now,  by   \cite[Theorem 2.4.14]{11} we see that $x_n\xrightarrow{w} 0$.
\end{proof}

\begin{proposition}\label{2.9}
Every Dunford-Pettis operator from a Banach lattice $E$ into a Banach space $X$ is b-weakly compact.
\end{proposition}
\begin{proof}
Let $T$ be a Dunford-Pettis operator from a Banach lattice $E$ into a Banach space $X$. By  \cite[Proposition 1]{4}, it is enough to show that  $\{Tx_n\}_n$ is norm convergent to zero for each b-order bounded disjoint sequence  $\{x_n\}_n$ in $E^+$.  Let  $\{x_n\}_n$ be a b-order bounded disjoint sequence in $E^+$.  As the canonical embedding of $E$ into $E^{\prime\prime}$ is a lattice homomorphism, $\{x_n\}_n$ is an order bounded disjoint sequence in $E^{\prime\prime}$. Thus by preceding lemma, $\{x_n\}_n$ is $\sigma(E, E^{\prime})$ convergent to zero in $E$. Now,  since $T$ is Dunford-Pettis, $\{Tx_n\}_n$ is norm convergent to zero. This completes  the proof.
\end{proof}
\vspace{0.3cm}
As a consequence of  \cite[Theorem 5.82]{1},  \cite[Theorem 2.2]{8} and  \cite[Theorem 2.3]{10}, we  have the following results.

\begin{corollary} \label{2.10}
Let $E$ be a Banach lattice and let $X$ be a Banach space. Then each b-weakly compact operator from $E$ into $X$ is Dunford-Pettis, if one of the following assertions is valid:
\begin{enumerate}
\item $E$ is an $AM$-space.
\item  The norm of $E^\prime$ is order continuous and $E$ has the Dunford-Pettis property (i.e. each weakly compact operator from a Banach space $E$ into another $F$ is Dunford-Pettis).
\end{enumerate}
\end{corollary}

For the next  results we need the following lemma:

\begin{lemma} \label{2.15}
\begin{enumerate}
\item  If an operator $T$ from a Banach space $X$ into a Banach space $Y$ is compact and $T(X)$ is closed, then $T(X)$ is finite-dimensional.  As a consequence,  if $T:X\rightarrow Y$ is a surjective compact operator between Banach spaces, then $Y$ is finite-dimensional.
\item  If $T:X\rightarrow Y$ is a weakly compact operator between Banach spaces and $T(X)$ is closed, then $T(X)$ is reflexive.  As a consequence,  if $T:X\rightarrow Y$ is a surjective weakly compact operator between Banach spaces,  then $Y$ is reflexive.
   \end{enumerate}
\end{lemma}
\begin{proof}
\begin{enumerate}
\item  Let $T:X\rightarrow Y$ be a compact operator between Banach spaces. Since $\overline{T(X)}=T(X)$, $T(X)$ is a Banach space. If $U$ denotes the open ball of $X$, then $T(U)$ is an open set in $T(X)$. On the other hand, $\overline{T(U)}$ is compact so, $T(X)$ is locally compact and hence $T(X)$ is finite-dimensional.
\item  Let $T:X\rightarrow Y$ be a weakly compact operator between Banach spaces. Since $T(X)$ is closed, $T(X)$ is a Banach space and from equality $T(X)=\bigcup_{n\in\mathbb{N}}nT(B_X)$, we see that $T(B_X)$ contains a closed ball of $T(X)$. On the other hand, $\overline{T(B_X)}$ is weakly compact, so, that closed ball is weakly compact, therefore $T(X)$ is reflexive.
\end{enumerate}
\end{proof}

 \begin{proposition}\label{2.16}
Let $E$ be a Banach lattice and let $X$ be a non-reflexive Banach space. If $T:E\rightarrow X$ is a surjective b-weakly compact operator, then the norm of $E^\prime$ is not order continuous.
\end{proposition}
\begin{proof}
If the norm of $E^\prime$ is order continuous then by using  \cite[Theorem 2.2]{8} and Lemma  \ref{2.15},  $X$ is reflexive which is a contradiction.
\end{proof}

\begin{proposition} \label{2.17}
Let $E$, $F$ be two Banach lattices and $F^\prime$ has order continuous norm. Suppose that $T:E \rightarrow F$ is an almost interval preserving, injective and b-weakly compact operator which has a closed range. Then $E$ is reflexive.  
\end{proposition}
\begin{proof}
 Since $T(E)$ is closed, $T(E)$ is a Banach space, so, $T_1 :E\rightarrow T(E)$ is a bijective operator between two Banach spaces. Then $T_1^\prime :T(E)^\prime \rightarrow E^\prime$ is bijective. Without lose generality we replace $T^\prime$ with $T_1^\prime$.  On the other hand, since $T$ is an almost interval preserving, by Theorem 1.4.19 \cite{11}, $T^\prime$ is a lattice homomorphism, and so by Theorem 2.15 \cite{1}, $T^\prime$ and ${(T')^{-1}}$ are both positive operators. 
  Since $F^\prime$ has order continuous norm, by \cite[Theorem 4.59]{1}, $F^\prime$ is a $KB$-space, and so  $T^\prime$ is b-weakly compact operator and  the norm of $E^\prime$ is order continuous. Since $T$ is a  b-weakly compact, by  \cite[Theorem 2.2]{8}, $T$ is weakly compact. So,  $T^\prime$ is weakly compact. Now, by Lemma \ref{2.15}, $E^\prime$ is reflexive and then $E$ is reflexive. This completes  the proof.  
\end{proof}
Recall that a nonzero element $x$ of a Banach lattice $E$ is discrete if the order ideal generated by $x$ is equal to the subspace generated by $x$. The vector lattice $E$ is discrete if it admits a complete disjoint system of discrete elements. For example the Banach lattice $\ell^2$ is discrete but $L^1[0,1]$ is not.

\begin{proposition}\label{2.18}
Let $E$ be a Banach lattice and let $X$ be a Banach space and let $T:E\rightarrow X$ be an injective b-weakly compact operator which its range is closed.  If one the following conditions holds, then $E$ is finite dimensional.
\begin{enumerate}
\item $E$ is an $AM$-space with order continuous norm.
\item $E$ is an $AM$-space and $E^\prime$ is discrete.
\end{enumerate}  
\end{proposition}
\begin{proof}
According to the proof of Proposition \ref{2.17}, $T^\prime:X^\prime \rightarrow E^\prime$ is a surjective operator. By \cite[Proposition 2.3]{10}, if one of the above conditions holds, then $T$ is a compact operator. Thus, $T^\prime$ is compact. Now, by Lemma \ref{2.15}, $E^\prime$ is finite dimensional. Hence, $E$ is finite dimensional. 
\end{proof}

\begin{definition}
An operator $T:E\rightarrow F$ between two normed vector lattices is unbounded $b$-weakly compact if $\{Tx_n\}$ is $un$-convergent for every positive increasing sequence $\{x_n\}_n$ in the closed unit ball $B_E$ of $E$.
\end{definition}
For normed vector lattices $E$ and $F$, the collection of unbounded $b$-weakly compact operators will be denoted by  $UW_b(E,F)$.
 If a Banach lattice $F$ has strong unit, by using Theorem 2.3  \cite{10a}, we have   $W_b(E,F)=UW_b(E,F)$.  
It is obvious that every $b$-weakly compact operator is unbounded $b$-weakly compact, but the following example shows that  the converse does not hold, in general.
   
\begin{example}\label{Ex:12}
Let $I_{c_0}$ be the identity mapping from $c_0$ into itself.  Then $I_{c_0}$ is an  unbounded $b$-weakly compact operator.  But $I_{c_0}$ is not a $b$-weakly compact operator.
\end{example}

The following characterization is obvious.
\begin{proposition}
Let $E$ and $F$ be two normed vector lattices and let $T$ be an operator from $E$ into $F$. Then the following are equivalent:
\begin{enumerate}
\item $T$ is unbounded $b$-weakly compact.
\item $\{Tx_n\}$ is norm convergent for each $b$-order bounded increasing sequence $\{x_n\}$ in $E^+$.
\end{enumerate}
\end{proposition}

It is easy to see that the class of unbounded $b$-weakly compact operators satisfies the domination property.
\begin{proposition}
Let $E$ and $F$ be two normed vector lattices and let $S$ and $T$ be two operators from $E$ into $F$ with $0\leq S\leq T$. If $T$ is unbounded $b$-weakly compact then $S$ is likewise unbounded $b$-weakly compact.
\end{proposition}

\begin{theorem}
Let $E$, $F$ be two Banach lattices and let $T\in L(E,F)$. If for an ideal $I$  of $E$ the restriction $T\vert_I:I\rightarrow F$ is a surjective homomorphism  which is also a $b$-weakly compact operator, then $T\in UW_b(E,F)$.  
\end{theorem}
\begin{proof}
Let $\{x_n\}$ be a positive increasing sequence in $E$ with $\sup_n\Vert x_n \Vert<\infty$ and let $x\in I$. Then $ x_n\wedge x\in I$, $ 0\leq x_n\wedge x\uparrow$ and $\sup_n\Vert x_n\wedge x \Vert\leq\Vert x_n\Vert<\infty$. Since $T\in W_b(I,F)$,  $\{T(x_{n}\wedge x)\}$  is convergent for each $x\in I$. As $T$ is homomorphism and surjective,  $\{T(x_{n})\wedge y\}$ is convergent for all $y\in F$ and proof follows.
\end{proof}

\begin{theorem}
Let $E$ and $F$ be two Banach lattices and let $T: E\rightarrow F$ be a surjective homomorphism. Then by one of the following conditions we have $T\in UW_b(E,F)$.
\begin{enumerate}
\item $E$ is a $KB$-space.
\item $F$ has order continuous norm.
\end{enumerate} 
\end{theorem}
\begin{proof}
If $E$ is a $KB$-space, then $T\in W_b(E,F)$ and proof follows.\\
Assume that $F$ has order continuous norm. Let $\{x_n\}\subset E^+$ be an increasing sequence such that $\sup\Vert x_n\Vert<\infty$ 
and let $x\in E^+$. Set $y_n=x_n\wedge x$, which follows that $y_n\uparrow\leq x$ and $\sup_n\Vert y_n\Vert\leq \Vert  x \Vert$. Since $T$ is lattice homomorphism, $T$ is positive, which follows  $Ty_n\uparrow\leq Tx$. By using Theorem 4.11 \cite{1}, $\{Ty_n\}$ is norm Cauchy, and so is norm convergence in $F$. On the other hand, $T(x_n\wedge x)=Tx_n\wedge Tx$. Therefore, $Tx_n\wedge y$ is norm convergent for each $y\in F$.
\end{proof}


\begin{thebibliography}{20}
\bibitem{1b} Y. Abramovich and G. Sirotkin,  \newblock{On order convergence of nets}, \newblock {\em Positivity.}  {\textbf 9} (2005),   287--292.
\bibitem{1} C.D. Aliprantis, O. Burkinshaw,  \newblock{Positive Operators}, Springer, Berlin, 2006.

\bibitem{2} S. Alpay,  B. Altin, C. Tonyali,  \newblock{On property (b) of vector lattices}, \newblock {\em  Positivity.} {\textbf 7} (2003), 135--139.
\bibitem{3} S. Alpay, B. Altin and C. Tonyali, \newblock { A note on Riesz spaces with property-b}, \newblock {\em Czechoslovak
Math. J.} {\textbf 56}  (2006),  765--772.

\bibitem{4} S. Alpay and B. Altin, \newblock { A note on b-weakly compact operators}, \newblock {\em Positivity.} {\textbf 11}  (2007), 575--582.
\bibitem{5} S. Alpay and Z. Ercan, \newblock {Characterizations of Riesz spaces with b-property}, \newblock {\em Positivity.} {\bf 13}   (2009), 21--30.

\bibitem{6} B. Altin, \newblock { Some properties of b-weakly compact operators}, \newblock {\em G.U. J. Science.} {\textbf 18}   (2005), 391--395,
\bibitem{7} B. Altin, \newblock {On b-weakly compact operators on Banach lattices}, \newblock {\em Taiwanese J. Math.} {\textbf 11} (2007), 143--150.

\bibitem{8} B. Aqzzouz,  A. Elbour, \newblock{On the weak compactness  of b-Weakly Compact Operators}, \newblock {\em Positivity.}  {\bf 14}   (2010), 75--81.

\bibitem{9} B. Aqzzouz,  A. Elbour, \newblock{Some Properties of the Class of b-Weakly Compact Operators}, \newblock {\em Complex Anal. Oper. Theory.} {\textbf 12} (2010), 1139--1145.

\bibitem{10} B. Aqzzouz, M. Moussa and J. Hmichane, \newblock{Some Characterizations of b-weakly compact operators}, \newblock{\em Math. Reports.} {\textbf 62}  (2010), 315--324. 
\bibitem{9d} Y. Deng,  M. $\text{O}^,$Brien and  V. G. Troitsky, \newblock{ Unbounded norm convergence in Banach lattices}, \newblock{\em Positivity.}  {\textbf 21}  (2017), 963--974.
\bibitem{9g} N. Gao and F. Xanthos, \newblock{Unbounded order convergence and application to martingales without probability},\newblock{\em J. Math. Anal. Appl.} {\textbf 415} (2014), 931--947.
\bibitem{11} P. Meyer-Nieberg, \newblock{Banach lattices}, Universitex. Springer, Berlin. MR1128093,  1991.
\bibitem{Megginson} R.E. Megginson, \newblock{ An Introduction to Banach space Theory}, Springer-Verlag New York. Inc, 1998.
\bibitem{12} H. Schaefer, \newblock{ lattices and positive operators}, Springer-Verlag, Berlin and New York, 1974.
\bibitem{10a} M. Kandic, M.A.A. Marabeh and  V. G. Troitsky, \newblock{ Unbounded norm topology in Banach lattices}, \newblock{\em J. Math. Anal. Appl.} {\bf 451}  (2017), 259--279.

\bibitem{13}  Z.L. Chen and  A.W. Wickstead,  \newblock{Vector lattices of weakly compact operators on Banach lattices},  \newblock {\em  Amer. Math. Soc.} {\bf 352}  (1999), 397--412. 

\end{thebibliography}
\end{document}